\documentclass[10pt]{amsart}
\usepackage{amssymb}
\usepackage{epsfig}
\usepackage{url}
\usepackage{setspace}
\usepackage{pdflscape}
\theoremstyle{plain}

\newtheorem{thm}{Theorem}[section]
\newtheorem{cor}[thm]{Corollary}
\newtheorem{lem}[thm]{Lemma}

\newtheorem{rem}[thm]{Remark}
\newtheorem{ques}[thm]{Question}
\newtheorem{conj}[thm]{Conjecture}

\renewcommand{\phi}{\varphi}

\newtheorem*{acknowledgement}{Acknowledgement}

\begin{document}

\title[On products of disjoint blocks]{On products of disjoint blocks of arithmetic progressions and related equations}
\author{Sz. Tengely}
\address{Institute of Mathematics\newline
 \indent University of Derecen\newline
 \indent P.O.Box 12\newline
 \indent 4010 Debrecen\newline
 \indent Hungary}
\email{tengely@science.unideb.hu}
\author{M. Ulas}
\address{Jagiellonian University\newline
\indent Faculty of Mathematics and Computer Science \newline
\indent Institute of Mathematics\newline
 \indent Łojasiewicza 6\newline
 \indent 30-348 Kraków\newline
 \indent Poland}
\email{Maciej.Ulas@im.uj.edu.pl}

\thanks{Research supported in part by the OTKA grants K100339 and NK101680.}

\begin{abstract}
In this paper we deal with Diophantine equations involving products of consecutive integers,
inspired by a question of Erdős and Graham.

\end{abstract}

\maketitle

\section{introduction}
Consider the polynomial
$$
f(x,k,d)=x(x+d)\cdots(x+(k-1)d).
$$
Erd\H{o}s \cite{er} and independently Rigge \cite{r} proved that $f(x,k,1)$ is never a perfect square.
A celebrated result of Erd\H{o}s and Selfridge \cite{ES} states
that $f(x,k,1)$ is never a perfect power if $x\geq 1$ and $k\geq 2.$ The literature of this type of Diophantine equations is very rich. First we mention some results related to the Diophantine equation
\begin{equation}\label{1}
f(x,k,d)=y^2.
\end{equation}
Euler proved (see \cite{Dickson} pp. 440 and 635) that a product of four terms in arithmetic progression is never a square.
Obl\'ath \cite{Oblath} obtained a similar statement for $k=5.$
Saradha and Shorey \cite{ss2} showed that \eqref{1} has no solutions with $k\geq 4$, provided that $d$ is a power of a prime number.
Laishram and Shorey \cite{omd4} extended this result to the case where either $d\leq 10^{10}$, or $d$ has at most six prime divisors.
Bennett, Bruin, Gy\H{o}ry and Hajdu \cite{BBGyH} solved \eqref{1} with $6\leq k\leq 11.$
Hirata-Kohno, Laishram, Shorey and Tijdeman \cite{HLST} completely solved \eqref{1} with $3\leq k<110$ and $x,d,k\geq 2, \gcd(x,d)=1.$

Many authors have studied the more general equation
\begin{equation}\label{2}
f(x,k,d)=by^l,
\end{equation}
where $b>0, l\geq 3$ and the greatest prime factor of $b$ does not exceed $k.$ Saradha \cite{Saradha_AA82} proved that \eqref{2}
has no solution with $1\leq d\leq 6,k\geq 4,\gcd(x,d)=1.$ Gy\H{o}ry \cite{Gyory_AA83} studied the cases $k=2,3,$ he determined all integral solutions.
Gy\H{o}ry, Hajdu and Saradha \cite{GyHS} showed that the product of four or five consecutive terms of an arithmetical progression of integers cannot be a perfect power, provided that the initial term is coprime to the difference.
Hajdu, Tengely and Tijdeman \cite{HTT} proved that the product of $k$ coprime integers in arithmetic progression cannot be a cube when $2 < k < 39.$
Gy\H{o}ry, Hajdu and Pint\'er \cite{GyHP2009} proved that for any positive integers $x, d$ and $k$ with $\gcd(x,d)=1$ and $3<k<35,$ the product $x(x+d)\cdots(x+(k-1)d)$ cannot be a perfect power.

Erd\H{o}s and Graham \cite{EG} asked if the Diophantine equation
$$
\prod_{i=1}^r f(x_i,k_i,1)=y^2
$$
has, for fixed $r\geq 1$ and $\{k_1,k_2,\ldots,k_r\}$ with $k_i\geq 4$ for $i=1,2,\ldots,r,$
at most finitely many solutions in positive integers $(x_1,x_2,\ldots,x_r,y)$
with  $x_i + k_i\leq x_{i+1}$ for $1\leq i\leq r-1.$
Ska{\l}ba \cite{Skalba} obtained a bound for the smallest solution and estimated the number of solutions below a given bound.
Ulas \cite{Ulas2005} answered the above question of Erd\H{o}s and Graham in the negative when either $r=4,(k_1,k_2,k_3,k_4)=(4,4,4,4)$ or $r\geq 6$ and $k_i=4, 1\leq i\leq r.$
Bauer and Bennett \cite{BB-EG} extended this result to the cases $r=3$ and $r=5.$
Bennett and Van Luijk \cite{BvL} constructed an infinite family of $r\geq 5$ non-overlapping blocks of five consecutive integers such that their product is always a perfect square.
Luca and Walsh \cite{LW} studied the case $(r,k_i)=(2,4).$

\section{Erdős-Graham type Diophantine problems}
In this section we present some related Diophantine equations involving products of consecutive integers. Let us recall that Bauer and Bennett \cite{BB-EG} proved that for each positive integer $j$ and a $j$ tuple $(k_{1},\ldots,k_{j})$ the Diophantine equation
\begin{equation}\label{fulleq}
y^2=x(x+1)\prod_{i=1}^{j}\prod_{l=0}^{k_{i}-1}(x_{i}+l)
\end{equation}
has infinitely many solutions in positive integers $x,x_{1},\ldots, x_{j}$. However, the proof they presented produces solutions which grow exponentially.
In the light of this result one can ask whether in some cases we can find solutions in polynomials with integer coefficients. In this direction we offer the following:

\begin{thm}\label{23squarethm}
The Diophantine equations
\begin{eqnarray}
x(x+1)y(y+1)(y+2)&=&z^2,\label{eq23square}\\
x(x+1)y(y+1)(y+2)(y+3)&=&z^2 \label{eq24square}
\end{eqnarray}
have infinitely many solutions in the ring $\mathbb{Z}[t]$. Moreover, the Diophantine equation
 \begin{equation}\label{eq25square}
x(x+1)y(y+1)(y+2)(y+3)(y+4)=z^2
\end{equation}
has at least two solutions in the ring $\mathbb{Z}[t]$.
\end{thm}

The next results deal with the question of whether the product of disjoint blocks of consecutive integers can be a product of two consecutive integers.
We thus consider the Diophantine equation
\begin{equation}\label{EGtype}
\prod_{i=1}^r f(x_i,k_i,d)=y(y+d)
\end{equation}

This question concerning the solvability in integers of the equation (\ref{EGtype}) can be seen as a variation on Erd\H{o}s and Graham question. We have the following:

\begin{thm}\label{EGvar}
The Diophantine equations
\begin{eqnarray}
x(x+1)y(y+1)(y+2)&=&z(z+1),\label{eq23triangle}\\
x(x+1)y(y+1)(y+2)(y+3)&=&z(z+1) \label{eq24triangle}
\end{eqnarray}
have infinitely many solutions in the ring $\mathbb{Z}[t]$. Moreover, for $k_{1}=3, r\geq 2$ and each $r-1$-tuple $k_{2},\ldots, k_{r}$ of positive integers the Diophantine equation
(\ref{EGtype}) has at least six solutions in the ring $\mathbb{Z}[x_{2},\ldots,x_{r}]$.
\end{thm}

\begin{thm}\label{2-2-4}
The Diophantine equation
\begin{equation}\label{eq224}
x(x+1)y(y+1)=z(z+1)(z+2)(z+3)
\end{equation}
has infinitely many solutions in positive integers satisfying the condition $(z-x)(z-x+2)\neq 0$.
\end{thm}

In \cite{Ulas2009} the second author proved that the system of Diophantine equations
\begin{equation*}
\begin{cases}
  \begin{array}{ccc}
    x(x+1)+y(y+1) & = & p(p+1) \\
    y(y+1)+z(z+1) & = & q(q+1) \\
    z(z+1)+x(x+1) & = & r(r+1) \\
  \end{array}
\end{cases}
\end{equation*}
has infinitely many solutions in integers satisfying the condition $0<x<y<z$. One can ask whether similar phenomenon holds for the multiplicative version of the above system.
More precisely: does the system of Diophantine equations
\begin{equation}\label{multsys}
\begin{cases}
  \begin{array}{ccc}
    x(x+1)y(y+1) & = & p(p+1) \\
    y(y+1)z(z+1) & = & q(q+1) \\
    z(z+1)x(x+1) & = & r(r+1) \\
  \end{array}
\end{cases}
\end{equation}
have infinitely many solutions in integers satisfying the condition $1<x<y<z$? Motivated by this question we prove the following:

\begin{thm}\label{MULT}
The system {\rm (\ref{multsys})} has infinitely many solutions in the ring of polynomials $\mathbb{Z}[t]$.
\end{thm}

Consider the Diophantine equations
\begin{equation}\label{z2}
(x-b)x(x+b)(y-b)y(y+b)=z^2
\end{equation}
and
\begin{equation}\label{z3}
(x-b)x(x+b)(y-b)y(y+b)=(z-b)z(z+b)
\end{equation}
where $b\in\mathbb{N}$ is a parameter. If a solution $(x,y,z)$ satisfies $b\mid x$ and $b\mid y,$ we call it trivial.
If $b=1,$ then Sastry \cite{Guy1994} noted that \eqref{z2} has infinitely many positive integer solutions
$(x,y,z),$ where $y = 2x-1$ and $(x + 1)(2x-1)$ is a square.
Zhang and Cai \cite{ZC} proved that there exist infinitely many nontrivial positive integer solutions of the Diophantine equation \eqref{z2} if $b\geq 2$ is an even integer.
They note that it is likely that for odd $b\geq 3$ integers there are also infinitely many solutions.
They showed that \eqref{z3} has infinitely many nontrivial positive integer solutions for $b=1$, and the set of rational solutions of it is dense in the set of real solutions for $b\geq 1.$
They posed the following question.
Are all the nontrivial positive integer solutions of \eqref{z3} for $b = 1$ with
$x \leq y$ given by $(F_{2n-1},F_{2n+1}, F_{2n}^2), n \geq 1$? We prove that all "large" solutions have this shape while "small" solutions belong to certain intervals.
We have the following statements.
\begin{thm}\label{ZC2}
Let $(x,y,z)$ be a nontrivial positive integer solution of equation \eqref{z2} and $k=y-x.$
Either
$$
x=-\frac{48 \, b^{2} k - 3 \, k^{3} \pm 2 \,
{\left(4 \, b^{2} - k^{2}\right)} \sqrt{-48
\, b^{2} + 3 \, k^{2}}}{6 \, {\left(16 \,
b^{2} - k^{2}\right)}}
$$
or $$1\leq x\leq \max_{1\leq i\leq 3} B_i,$$
where
\begin{eqnarray*}
B_1&=&2\max\left|-6 \, b^{2} k^{2} + \frac{3}{8} \, k^{4} \pm \frac{3}{2} \, k\right|,\\
B_2&=&2\max\left|-6 \, b^{2} k^{3} + \frac{3}{8} \, k^{5} \mp b^{2} \pm \frac{3}{8} \, k^{2}\right|^{1/2},\\
B_3&=&2\max\left|-b^{4} k^{2} - \frac{1}{4} \, b^{2} k^{4} -\frac{1}{64} \, k^{6} \pm \frac{1}{4} \, b^{2}k \pm \frac{1}{32} \, k^{3} - \frac{1}{64}\right|^{1/3}.
\end{eqnarray*}
\end{thm}
\begin{cor}\label{cor1}
If $3\leq b\leq 13,$ $b$ is odd and $2b< k\leq 300,$ then all nontrivial positive integer solutions of equation \eqref{z2}
are as follows
{\tiny
\begin{center}
\begin{tabular}{|l|l|l|l|l|l|} \hline
$b$ & $(x,y,z)$ & $b$ & $(x,y,z)$
& $b$ & $(x,y,z)$ \\ \hline \hline
$3$ & $\left(5, 12,
360\right)$ & $5$ &
$\left(145, 343,
11083800\right)$ & $7$ &
$\left(250, 507,
45103500\right)$ \\ \hline
$3$ & $\left(7, 18,
1260\right)$ & $5$ &
$\left(33, 280, 877800\right)$
& $9$ & $\left(15, 36,
9720\right)$ \\ \hline
$3$ & $\left(4, 21,
504\right)$ & $5$ &
$\left(16, 275, 277200\right)$
& $9$ & $\left(21, 54,
34020\right)$ \\ \hline
$3$ & $\left(8, 33, 3960\right),
\left(35, 60, 95760\right)$ & $7$
& $\left(10, 27,
3060\right)$ & $9$ &
$\left(12, 63, 13608\right)$ \\
\hline
$3$ & $\left(10, 42,
8190\right)$ & $7$ &
$\left(105, 128, 1552320\right)$
& $9$ & $\left(24, 99,
106920\right), \left(105, 180,
2585520\right)$ \\ \hline
$3$ & $\left(7, 45,
5040\right)$ & $7$ &
$\left(8, 42, 2940\right), \left(41,
75, 167280\right)$ & $9$ &
$\left(11, 90, 17820\right)$ \\
\hline
$3$ & $\left(32, 87,
146160\right)$ & $7$ &
$\left(34, 75, 125460\right)$
& $9$ & $\left(30, 126,
221130\right)$ \\ \hline
$3$ & $\left(93, 245,
3437280\right)$ & $7$ &
$\left(9, 56, 7056\right)$ &
$9$ & $\left(21, 135,
136080\right)$ \\ \hline
$3$ & $\left(125, 363,
9662400\right)$ & $7$ &
$\left(32, 91, 152880\right)$
& $9$ & $\left(25, 153,
220320\right)$ \\ \hline
$3$ & $\left(77, 333,
4102560\right)$ & $7$ &
$\left(13, 98, 38220\right)$
& $9$ & $\left(10, 171,
30780\right)$ \\ \hline
$5$ & $\left(7, 30,
2100\right)$ & $7$ &
$\left(42, 128, 388080\right)$
& $9$ & $\left(96, 261,
3946320\right)$ \\ \hline
$5$ & $\left(11, 49,
11088\right)$ & $7$ &
$\left(8, 105, 11760\right)$
& $11$ & $\left(91, 119,
1113840\right)$ \\ \hline
$5$ & $\left(6, 49,
2772\right)$ & $7$ &
$\left(8, 128, 15840\right)$
& $11$ & $\left(13, 132,
37752\right)$ \\ \hline
$5$ & $\left(11, 55,
13200\right)$ & $7$ &
$\left(12, 140, 55860\right)$
& $11$ & $\left(12, 253,
66792\right)$ \\ \hline
$5$ & $\left(6, 55, 3300\right),
\left(21, 70, 54600\right)$ & $7$
& $\left(18, 169,
154440\right)$ & $13$ &
$\left(22, 77, 55440\right)$ \\
\hline
$5$ & $\left(7, 75,
8400\right)$ & $7$ &
$\left(32, 189, 458640\right)$
& $13$ & $\left(14, 169,
42588\right)$ \\ \hline
$5$ & $\left(19, 100,
79800\right)$ & $7$ &
$\left(11, 169, 61776\right)$
& $13$ & $\left(15, 182,
70980\right)$ \\ \hline
$5$ & $\left(3605, 3703,
48773919600\right)$ & $7$ &
$\left(185, 363,
17387040\right)$ & $13$ &
$\left(99, 288, 4767840\right)$
\\ \hline
\end{tabular}
\end{center}
\begin{center}{\rm Table 1}\end{center}}
\end{cor}

\begin{rem}
We computed all nontrivial solutions with $3\leq b\leq 25,$ $b$ is odd and $2b< k\leq 300.$ There are 144
such solutions, the list can be downloaded from \url{http://math.unideb.hu/media/tengely-szabolcs/XblockYblockZ2.txt}.
We note that the total running time of our calculations was 11.6 hours on an Intel Core i5 2.6GHz PC.
\end{rem}
\begin{thm}\label{Fib}
Let $(x,y,z)$ be a nontrivial positive integer solution of equation \eqref{z3} with $b=1$. Either
$$
(x,y,z)=(F_{2n-1},F_{2n+1}, F_{2n}^2)\quad\mbox{for some } n\geq 1
$$
or
$$
1\leq x\leq -\frac{1}{2} \, k + \frac{1}{2} \,\sqrt{3 \, k^{2} + 2 \, k\sqrt{k^{2} + 4} +4},
$$
where $k=y-x.$
\end{thm}
Based on the previous theorem we have the following numerical result.
\begin{cor}
If $4\leq k\leq 5000,$ then all nontrivial positive integer solutions of equation \eqref{z3} with $b=1$ have
the form $(x,y,z)=(F_{2n-1},F_{2n+1}, F_{2n}^2)\quad\mbox{for some } n\geq 1.$
\end{cor}

\section{proofs of the results}
Before we present the proofs let us note that if $(Z',X')$ is a solution of the Diophantine equation $Z^2-AX^2=B$ and $(Z,X)$ is s solution of $Z^2-AX^2=1$ with $X\neq 0$, then for each $n$ the pair $(Z_{n},X_{n})$, where
$$
Z_{0}=Z' \quad X_{0}=X',\quad Z_{n}=Z\cdot Z_{n-1}+AX\cdot X_{n-1},\quad X_{n}=X\cdot Z_{n-1}+Z\cdot X_{n-1}
$$
is solution of $Z^2-AX^2=B$.

In order to shorten the notation we write
$$
f_{k}(x):=f(x,k,1)=x(x+1)\cdot\ldots\cdot (x+k-1).
$$
\begin{proof}[Proof of Theorem \ref{23squarethm}]
We observe that the equation (\ref{eq23square}) can be rewritten in the following form
\begin{equation}\label{pelleq23square}
Z^2-f_{3}(y)X^2=-f_{3}(y).
\end{equation}
with $Z=2z$ and $2x+1=X$. In order to solve this equation we take $y=t^2+1$ and we observe that the equation (\ref{pelleq23square}) has the solution
$$
Z'=tf_{3}(t^2+1),\quad X'=t^4+3t^2+1.
$$
Moreover, we note that the Diophantine equation $Z^2-f_{3}(t^2+1)X^2=1$ has the nontrivial solution
$$
Z=t^4+3t^2+1,\quad X=t.
$$
According to remark given at the beginning of the section we see that for each $n$ the pair of polynomials $(Z_{n},X_{n})$ defined by the recurrence relations
\begin{equation*}
\begin{cases}
\begin{array}{lll}
  Z_{0} & = & t f_{3}(t^2+1), \\
  X_{0} & = & t^4+3t^2+1 ,\\
  Z_{n} & = & (t^4+3t^2+1)Z_{n-1}+t f_{3}(t^2+1)X_{n-1},\\
  X_{n} & = & t Z_{n-1}+(t^4+3t^2+1)X_{n-1},
\end{array}
\end{cases}
\end{equation*}
is a solution of the equation (\ref{pelleq23square}). It is clear from the definition that $Z_{n}, X_{n}\in\mathbb{Z}[t]$ for each $n\in\mathbb{N}$. Moreover, by simple induction on $n$ we check that $X_{n}(2t)\equiv 1\pmod{2}$ and $Z_{n}(2t)\equiv 0\pmod{2}$ in the ring of polynomials $\mathbb{Z}[t]$. As a consequence we get that for each $n$ the pair of polynomials
$$
x_{n}(t)=\frac{1}{2}(X_{n}(2t)-1),\quad z_{n}(t)=\frac{1}{2}Z_{n}(2t)
$$
is the solution of equation (\ref{eq23square}) in the ring $\mathbb{Z}[t]$.

\bigskip

In order to get the polynomial solutions of the equation (\ref{eq24square}) we use the same method as above. We take $y=t$, where $t$ is a variable and we rewrite our equation in the form
\begin{equation}\label{pelleq24square}
Z^2-f_{4}(t)X^2=-f_{4}(t)
\end{equation}
with $Z=2z$ and $X=2x+1$. We found that
$$
Z'=f_{4}(t),\quad X'=t^2 + 3t + 1
$$
is a solution of the equation (\ref{pelleq24square}) and the pair
$$
Z=t^2 + 3t + 1,\quad X=1
$$
solves the equation $Z^2-f_{4}(t)X^2=1$. We thus see that for each $n\in\mathbb{N}$ the pair of polynomials $(Z_{n},X_{n})$ defined by the recurrence relations
\begin{equation*}
\begin{cases}
\begin{array}{lll}
  Z_{0} & = & f_{4}(t), \\
  X_{0} & = & t^2+3t+1 ,\\
  Z_{n} & = & (t^2+3t+1)Z_{n-1}+f_{4}(t)X_{n-1},\\
  X_{n} & = & Z_{n-1}+(t^2+3t+1)X_{n-1},
\end{array}
\end{cases}
\end{equation*}
is a solution of the equation (\ref{pelleq24square}) in the ring $\mathbb{Z}[t]$. Similarly as in the previous case one can easily check that $X_{n}(2t)\equiv 1\pmod{2}$ and $Z_{n}(2t)\equiv 0\pmod{2}$ in $\mathbb{Z}[t]$ and in consequence, for each $n$ the pair of polynomials with integer coefficients
$$
x_{n}(t)=\frac{1}{2}(X_{n}(2t)-1),\quad z_{n}(t)=\frac{1}{2}Z_{n}(2t)
$$
is the solution of the equation (\ref{eq24square}).

\bigskip
Finally, in order to show that the equation (\ref{eq25square}) has polynomial solutions we performed numerical search and found that triplets of  polynomials
\begin{equation*}
\begin{array}{lll}
x=2t(t+1)(2t-1)(2t+3),&y=4t^2+4t-3,& z=2x(y+2)(2t+1)(2t^2+2t-1),\\
x=(2t^2+2t+1)(4t^2+4t+5),&y=(2t+1)^2,& z=4x(y+2)(2t+1)(t^2+t+1)
\end{array}
\end{equation*}
satisfy the equation (\ref{eq25square}).
\end{proof}

\begin{proof}[Proof of Theorem \ref{EGvar}]
We proceed in the same way as in the proof of Theorem \ref{23squarethm}. This time we take $y=4t^2+1$, where $t$ is a variable. In this situation our equation (\ref{eq23triangle}) is equivalent with the following one:
\begin{equation}\label{pelleq23triangle}
Z^2-f_{3}(4t^2+1)X^2=1-f_{3}(4t^2+1),
\end{equation}
where $Z=2z+1$ and $X=2x+1$. We found that the pair of polynomials
$$
Z'=128t^7+192t^5-16t^4+88t^3-12t^2+12t-1,\quad X'=16t^4+12t^2-2t+1
$$
satisfies the equation (\ref{pelleq23triangle}). Moreover, the pair
$$
Z=16t^4+12t^2+1,\quad X=2t
$$
satisfies the corresponding equation $Z^2-f_{3}(4t^2+1)X^2=1$. As a consequence we see that for each $n\in\mathbb{N}$ the pair of polynomials $(Z_{n},X_{n})$ defined by the recurrence relations
\begin{equation*}
\begin{cases}
\begin{array}{lll}
  Z_{0} & = & 128t^7+192t^5-16t^4+88t^3-12t^2+12t-1, \\
  X_{0} & = & 16t^4+12t^2-2t+1 ,\\
  Z_{n} & = & (16t^4+12t^2+1)Z_{n-1}+2tf_{3}(4t^2+1)X_{n-1},\\
  X_{n} & = & 2tZ_{n-1}+(16t^4+12t^2+1)X_{n-1},
\end{array}
\end{cases}
\end{equation*}
is a solution of the equation (\ref{pelleq24square}) in the ring $\mathbb{Z}[t]$. A simple induction shows that for each $n\in\mathbb{N}$ we have $X_{n}Z_{n}\equiv 1\pmod{2}$ in the ring $\mathbb{Z}[t]$ and thus the pair
$$
x_{n}=\frac{1}{2}(X_{n}-1),\quad z_{n}=\frac{1}{2}(Z_{n}-1)
$$
is the solution of the equation (\ref{eq23triangle}) with $y=4t^2+1$.

\bigskip

We consider now the equation (\ref{eq24triangle}) with $y=t$. It is equivalent with the following one:
\begin{equation}\label{pelleq24triangle}
Z^2-f_{4}(t)X^2=1-f_{4}(t),
\end{equation}
where $Z=2z+1$ and $X=2x+1$. We found that the pair of polynomials
$$
Z'=2 t^6+18 t^5+58 t^4+78 t^3+36 t^2-1,\quad X'=2 t^4+12 t^3+20 t^2+6 t-1
$$
satisfies the equation (\ref{pelleq24triangle}). Moreover, the pair
$$
Z=t^2 + 3t + 1,\quad X=1
$$
satisfies the corresponding equation $Z^2-f_{4}(t)X^2=1$. As a consequence we see that for each $n\in\mathbb{N}$ the pair of polynomials $(Z_{n},X_{n})$ defined by the recurrence relations
\begin{equation*}
\begin{cases}
\begin{array}{lll}
  Z_{0} & = & 2 t^6+18 t^5+58 t^4+78 t^3+36 t^2-1, \\
  X_{0} & = & 2 t^4+12 t^3+20 t^2+6 t-1,\\
  Z_{n} & = & (t^2 + 3t + 1)Z_{n-1}+f_{4}(t)X_{n-1},\\
  X_{n} & = & Z_{n-1}+(t^2 + 3t + 1)X_{n-1},
\end{array}
\end{cases}
\end{equation*}
is a solution of the equation (\ref{pelleq24triangle}) in the ring $\mathbb{Z}[t]$. A simple induction shows that for each $n\in\mathbb{N}$ we have $X_{2n}(2t+1)Z_{2n}(2t+1)\equiv 1\pmod{2}$ in the ring $\mathbb{Z}[t]$ and thus the pair
$$
x_{n}=\frac{1}{2}(X_{2n}(2t+1)-1),\quad z_{n}=\frac{1}{2}(Z_{2n}(2t+1)-1)
$$
is the solution of the equation (\ref{eq24triangle}) with $y=t$.
\bigskip

Finally, in order to prove the last statement of our theorem let us put
$$
A=\prod_{i=2}^r f(x_i,k_i,d)
$$
and consider the curve
$$
C:\;Ax(x+d)(x+2d)=y(y+d).
$$
From geometric point of view $C$ can be seen as a genus one curve defined over rational function field $\mathbb{Q}(A,d)$. The Weierstrass equation for $C$ is given by
$$
C':\;Y^2=X^3 + 12AdX^2 + 32A^2d^2X + 16A^2d^2,
$$
where the corresponding maps are the following:
$$
\begin{array}{ccl}
  \phi:\;C\ni (x,y) & \mapsto & (X,Y)=(4Ax,4A(2y+d))\in C', \\
  \phi^{-1}:\;C'\ni (X,Y) & \mapsto & (x,y)=\left(\frac{X}{4A},\frac{Y-4Ad}{8A}\right)\in C.
\end{array}
$$
Now using the trivial points with $y=0$ lying on $C$ we can define the points
\begin{eqnarray*}
P_{1}&=&\phi((0,0))=(0,4Ad),\\
P_{2}&=&\phi((-d,0))=(-4Ad,4Ad),\\
P_{3}&=&\phi((-2d,0))=(-8Ad,4Ad).
\end{eqnarray*}
One can easily check that for each $i,j\in\{1,2,3\},i\neq j$ the points $2P_{i}$ and $2P_i+2P_j$ have polynomials with integer coefficients as coordinates and the same is true
for the points $\phi^{-1}(2P_{i})$ and $\phi^{-1}(2P_{i}+2P_j).$ This leads us to the solutions of the equation defining $C$:
\begin{equation*}
\begin{array}{ll}
  x=Ad^2 - d, & y=-A^2d^3\\
  x=Ad^2 - d, & y=A^2d^3 - d, \\
  x=4Ad^2 + d, & y=8A^2d^3 + 6Ad^2\\
  x=4Ad^2 + d, & y=-8A^2d^3 - 6Ad^2 - d\\
  x=4Ad^2 - 3d, & y=8A^2d^3 - 6Ad^2\\
  x=4Ad^2 - 3d,  & y=-8A^2d^3 + 6Ad^2 - d.
\end{array}
\end{equation*}
From the definition of $A$ we know that it is essentially a polynomial in $\mathbb{Z}[x_{2},\ldots,x_{r}]$ and hence we get the statement of our theorem.

\end{proof}

\begin{proof}[Proof of Theorem \ref{2-2-4}]
In order to get the statement of our theorem we consider the intersection of the surface, say $S$, defined by the equation (\ref{eq224}) and the plane $L$ defined by the equation
$$
L:\;x+y=4z+5.
$$
We then observe that $S\cap L=C_{1}\cup C_{2}$, where
\begin{align*}
&C_{1}:\;(2z-4x+1)^2-3(2x+1)^2=-2,\\
&C_{2}:\;(2z+4x+5)^2-5(2x+1)^2=-4.
\end{align*}
Using standard methods we find that all solutions in positive integers of corresponding Pell type equations $U^2-3V^2=-2$ and $U'^2-5V'^2=-4$ are
\begin{equation*}
\begin{array}{lll}
U_{0}=1, V_{0}=1,&  U_{n+1}=2U_{n}+3V_{n}, & V_{n+1}=U_{n}+2V_{n},\\
U'_{0}=1, V'_{0}=1,&  U'_{n+1}=9U'_{n}+20V'_{n},& V'_{n+1}=4U'_{n}+9V'_{n}
\end{array}
\end{equation*}
respectively. One can easily check, by induction on $n$, that $V_{n}V'_{n}\equiv 1\pmod{2}$ and $U_{n}U'_{n}\equiv 1\pmod{2}$ and in consequence, for each $n\in\mathbb{N}$ the triplets
\begin{equation*}
\begin{array}{lll}
  x_{n}=\frac{1}{2}(V_{n}-1),  & y_{n}=\frac{1}{2}(4U_{n}+7V_{n}-1), & z_{n}=\frac{1}{2}(U_{n}+2V_{n}-3), \\
  x_{n}=\frac{1}{2}(V'_{n}-1), & y_{n}=\frac{1}{2}(9V'_{n}-4U_{n}-1),& z_{n}=\frac{1}{2}(U'_{n}-2V'_{n}-3)
\end{array}
\end{equation*}
are non-trivial solutions in non-negative integers of the equation (\ref{eq224}).
\end{proof}
\begin{rem}{\rm
Without much of work one can find that the related Diophantine equation
$$
x(x+1)y(y+1)=z(z+1)(z+2)
$$
has infinitely many solutions in positive integers satisfying the condition $x+1<y, (y-z)(y-z-1)\neq 0$. In fact, the above equation has polynomial solutions of the following form:
\begin{equation*}
\begin{array}{lll}
  x=t,    & y=t^2+t-2,   & z=(t-1)(t+2), \\
  x=t,    & y=t^2+t+1,   & z=t(t+1), \\
  x=8t+3, & y=8t^2+7t+1, & z=2(8t^2+7t+1), \\
  x=8t+4, & y=8t^2+9t+2, & z=2(8t^2+9t+2).
\end{array}
\end{equation*}
}
\end{rem}

\begin{proof}[Proof of Theorem \ref{MULT}]
In fact we prove a slightly stronger result, i.e. that the system (\ref{multsys}) has infinitely many polynomial solutions with $x=t$. In order to do that let us observe that the first equation from the system (\ref{multsys}) is equivalent with the following one:
\begin{equation}\label{firsteq}
P^2-t(t+1)Y^2=1-t(t+1),
\end{equation}
with $P=2p+1$ and $Y=2y+1$. This equation has infinitely many solutions in polynomials $P,Y\in\mathbb{Z}[t]$. Indeed, the equation (\ref{firsteq}) is satisfied by $P=1, Y=1$ and the related equation $P^2-t(t+1)Y^2=1$ has the solution
$$
P'=2t+1,\quad Y'=2.
$$
As a consequence we see that for each $n\in\mathbb{N}$ the pair $(P_{n},Y_{n})$ of polynomials defined by the recurrence relations
\begin{equation*}
\begin{cases}
\begin{array}{lll}
  P_{0} & = & 1 \\
  Y_{0} & = & 1 \\
  P_{n} & = & (2t+1)P_{n-1}+2t(t+1)Y_{n-1} \\
  Y_{n} & = & 2P_{n-1}+(2t+1)Y_{n-1}
\end{array}
\end{cases}
\end{equation*}
is a solution of the equation (\ref{firsteq}). We thus see that the polynomials
\begin{equation*}
\begin{array}{ll}
  y=y_{n}=\frac{1}{2}(Y_{n}-1) & p=p_{n}=\frac{1}{2}(P_{n}-1) \\
  z=z_{n}=y_{n+1} & r=r_{n}=p_{n+1}
\end{array}
\end{equation*}
satisfy the first and third equation in the system (\ref{multsys}). In order to get the result it is enough to prove that with our choice of $x, y, z$ the second equation in the system (\ref{multsys}) is satisfied too, i.e. $y(y+1)z(z+1)=y_{n}(y_{n}+1)y_{n+1}(y_{n+1}+1)=q(q+1)$ for some $q\in\mathbb{Z}[t]$. This is easy due to the identity
\begin{align*}
f\left(\frac{u-1}{2}\right)&f\left(\frac{2v+(2t+1)u-1}{2}\right)-f\left(\frac{(v-u)((2t^2+4t+1)u+(2t+3)v)}{4(t^2+t-1)}\right)\\
                           &=(v^2-t(t+1)u^2-1+t(t+1))H(u,v),
\end{align*}
where $4(t^2+t-1)H$ is a polynomial in $\mathbb{Z}[u,v,t]$ and $f(x)=x(x+1)$. If we put now $u=Y_{n}, v=P_{n}$ then we have the equalities
$$
y_{n}=\frac{1}{2}(u-1),\quad y_{n+1}=\frac{2v+(2t+1)u-1}{2},
$$
and simple induction reveals that
$$v-u=P_{n}-Y_{n}\equiv 0\pmod{2(t^2+t-1)}\quad\mbox{and}\quad uv=Y_{n}P_{n}\equiv 1\pmod{2}
$$
in the ring $\mathbb{Z}[t]$. As a consequence of our reasoning we see that for each $n\in\mathbb{N}$ the function
$$
q_{n}=\frac{(P_n-Y_n)((2t^2+4t+1)Y_n+(2t+3)P_n)}{4(t^2+t-1)}
$$
is a polynomial in $\mathbb{Z}[t]$ and thus we get the result.
\end{proof}

\begin{proof}[Proof of Theorem \ref{ZC2}]
In the proof we will use the following result of Fujiwara \cite{Fujiwara}.
\begin{lem}\label{Fujiwara}
Put $p(z)=\sum_{i=0}^na_iz^i, a_n\neq 0,$ where $a_i\in\mathbb{R}$ for all $i=0,1,\ldots,n.$ Then
$$
\max\{|\zeta|: p(\zeta)=0\}\leq 2\max\left\{\left|\frac{a_{n-1}}{a_n}\right|,\left|\frac{a_{n-2}}{a_n}\right|^{1/2},\ldots,\left|\frac{a_0}{2a_n}\right|^{1/n}\right\}.
$$
\end{lem} We apply Runge's method to determine a bound for the size of integral solutions.
Let $F(x)=(x-b)x(x+b)(x+k-b)(x+k)(x+k+b).$
The polynomial part of the Puiseux expansion of
$$
\left((x-b)x(x+b)(x+k-b)(x+k)(x+k+b)\right)^{1/2}
$$
is
$$P(x)=x^{3} + \frac{3}{2} \, k x^{2} + \left(-b^{2}
+ \frac{3}{8} \, k^{2}\right) x - \frac{1}{2}
\, b^{2} k - \frac{1}{16} \, k^{3}.$$
We have that
\begin{eqnarray*}
&&256F(x)-(16P(x)-1)^2=32 x^{3} + \left(-192 \, b^{2} k^{2} + 12 \,
k^{4} + 48 \, k\right) x^{2}+\\
&&+ \left(-192 \,b^{2} k^{3} + 12 \, k^{5} - 32 \, b^{2} + 12
\, k^{2}\right) x - 64 \, b^{4} k^{2} - 16 \,
b^{2} k^{4} - k^{6} - 16 \, b^{2} k - 2 \,
k^{3} - 1,\\
&&256F(x)-(16P(x)+1)^2=-32 x^{3} + \left(-192 \, b^{2} k^{2} + 12 \,
k^{4} - 48 \, k\right) x^{2}+\\
&& + \left(-192 \,b^{2} k^{3} + 12 \, k^{5} + 32 \, b^{2} - 12
\, k^{2}\right) x - 64 \, b^{4} k^{2} - 16 \,
b^{2} k^{4} - k^{6} + 16 \, b^{2} k + 2 \,
k^{3} - 1.
\end{eqnarray*}
Fujiwara's result implies that all roots of these cubic polynomials satisfy
$|x|\leq\max_{1\leq i\leq 3}B_i,$ where
\begin{eqnarray*}
B_1&=&2\max\left|-6 \, b^{2} k^{2} + \frac{3}{8} \, k^{4} \pm \frac{3}{2} \, k\right|,\\
B_2&=&2\max\left|-6 \, b^{2} k^{3} + \frac{3}{8} \, k^{5} \mp b^{2} \pm \frac{3}{8} \, k^{2}\right|^{1/2},\\
B_3&=&2\max\left|-b^{4} k^{2} - \frac{1}{4} \, b^{2} k^{4} -\frac{1}{64} \, k^{6} \pm \frac{1}{4} \, b^{2}k \pm \frac{1}{32} \, k^{3} - \frac{1}{64}\right|^{1/3}.
\end{eqnarray*}
Therefore if $|x|>\max_{1\leq i\leq 3}B_i,$ then
either
$$
(16P(x)+1)^2<256F(x)=(16y)^2<(16P(x)-1)^2
$$
or
$$
(16P(x)-1)^2<256F(x)=(16y)^2<(16P(x)+1)^2.
$$
Hence $y=\pm P(x).$ It remains to solve the equation $F(x)=P(x)^2.$ It follows that
$$
x=-\frac{48 \, b^{2} k - 3 \, k^{3} \pm 2 \,
{\left(4 \, b^{2} - k^{2}\right)} \sqrt{-48
\, b^{2} + 3 \, k^{2}}}{6 \, {\left(16 \,
b^{2} - k^{2}\right)}}.
$$
\end{proof}
\begin{proof}[Proof of Theorem \ref{Fib}]
We apply Runge's method to determine an upper bound for the size of possible positive integer solutions
of the equation
$$
F(x):=(x-1)x(x+1)(x+k-1)(x+k)(x+k+1)=(z-1)z(z+1),
$$
where $y=x+k$ for some positive integer $k.$
We have that
$$
(x^2+kx-1)^3<F(x)<(x^2+kx)^3
$$
if $x$ is large. In fact, the second inequality is true if $k>1.$ The roots of the polynomial $F(x)-(x^2+kx-1)^3$
are as follows
\begin{eqnarray*}
-\frac{1}{2} \, k - \frac{1}{2} \,\sqrt{3 \, k^{2} + 2 \, k\sqrt{k^{2} + 4} +4}&\approx&-\frac{1}{2} \, k {\left(\sqrt{5} + 1\right)},\\
-\frac{1}{2} \, k - \frac{1}{2} \,\sqrt{3 \, k^{2} - 2 \, k\sqrt{k^{2} + 4} +4}&\approx&-k,\\
-\frac{1}{2} \, k + \frac{1}{2} \,\sqrt{3 \, k^{2} - 2 \, k\sqrt{k^{2} + 4} +4}&\approx&\frac{1}{k^{3}},\\
-\frac{1}{2} \, k + \frac{1}{2} \,\sqrt{3 \, k^{2} + 2 \, k\sqrt{k^{2} + 4} +4}&\approx&\frac{1}{2} \, k {\left(\sqrt{5} - 1\right)}.
\end{eqnarray*}
Therefore if $$x>-\frac{1}{2} \, k + \frac{1}{2} \,\sqrt{3 \, k^{2} + 2 \, k\sqrt{k^{2} + 4} +4},$$
then the first inequality is valid.
Similarly we obtain that
$$
(z-1)^3<(z-1)z(z+1)<(z+1)^3
$$
if $z\notin\{-1,1\}.$
Assume that $x>-\frac{1}{2} \, k + \frac{1}{2} \,\sqrt{3 \, k^{2} + 2 \, k\sqrt{k^{2} + 4} +4}$ and $z\notin\{-1,1\}.$
We obtain that
$$
(x^2+kx-1)^3-(z+1)^3<0<(x^2+kx)^3-(z-1)^3.
$$
It follows that $z=x^2+kx-1$ or $z=x^2+kx.$
If  $z=x^2+kx,$ then $(k^2 + 2kx + 2x^2 - 2)(k + x)x=0$ and we get that either $x=0,x=-k$ or $|k|\leq 2.$
In the latter case $k=1$ or 2 and we obtain overlapping blocks, a contradiction.

If $z=x^2+kx-1,$ then $(k^2 - kx - x^2 + 1)(k + x)x=0$ and we have that $x=0,x=-k$ or
$$
x=-\frac{1}{2}k\pm \frac{1}{2}\sqrt{5k^2+4}.
$$
Since $x$ is a positive integer it follows that $k=F_{2n}.$
It yields that either $x=F_{2n-1}$ or $x=-F_{2n+1}.$ The latter is negative so the only possible
positive solution is $x=F_{2n-1}.$ Since $y=x+k,$ we obtain that $y=F_{2n+1}.$ Thus $(x,y,z)=(F_{2n-1},F_{2n+1},F_n^2)$ provides solutions.
\end{proof}

\begin{proof}[Proof of Corollary \ref{cor1}]
We wrote a Sage \cite{sage} code to determine all integral solution of equation \eqref{z3} with $b=1$
in the interval provided by Theorem \ref{Fib}.
\end{proof}

\section{Some additional remarks and questions}\label{questions}

In this final section we collect some additional remarks and questions which we were unable to answer during our research.
\begin{ques}
Does the equation (\ref{eq23triangle}) have infinitely many polynomial solutions which are not in the sequence constructed in the proof of Theorem \ref{EGvar}?
\end{ques}
\begin{rem}{\rm
The above question is motivated by the observation of the existence of polynomial solutions of (\ref{eq23triangle}) which are not contained in the family we constructed. Indeed, we have the following solutions:
\begin{equation*}
\begin{array}{lll}
  x=t, & y=t^2+t-1,  &z=(t^2+t-1) (t^2+t+1),\\
  x=t, & y=(2t+1)^2-4, &z=2 t (t+1) (2 t-1) (2 t+3), \\
  x=t, & y=(2t+1)^2, &z=2 t (t+1) (4 t^2+4 t+3), \\
  x=t(8t^3-6t-1), & y=4t^2-3, &z=(2 t-1) (2 t+1) (2 t^2-1) (8 t^3-6 t-1).
\end{array}
\end{equation*}
}\end{rem}
\bigskip

The next question is motivated by the result of Sastry mentioned in \cite[D17]{Guy1994} which says that the Diophantine equation $x(x+1)(x+2)y(y+1)(y+2)=z^2$ has infinitely many solutions in integers satisfying  $y>x+2$. In this direction one can ask the following:

\begin{ques}
Does the equation
$$
z^2=\frac{x(x+1)(x+2)}{y(y+1)(y+2)}
$$
have infinitely many solutions in positive integers satisfying $x\neq y$?
\end{ques}
This seems to be a difficult question. In the range $x<10^7, y<10^5$ we found only 10 solutions given in the Table 2 below:

\begin{center}
\begin{tabular}{|l|l|l|c|l|l|l|}
  \hline
  $x$ & $y$ & $z$    &  & $x$ & $y$ & $z$ \\
  \hline
  2 & 1 & 2    &  & 1680 & 5 & 4756 \\
  14 & 5 & 4   &  & 1680 & 14 & 1189 \\
  26 & 12 & 3  &  & 13454 & 90 & 1798 \\
  48 & 1 & 140 &  & 57120 & 168 & 6214 \\
  48 & 2 & 70  &  & 114242 & 337 & 6214 \\
  \hline
\end{tabular}
\begin{center}Table 2\end{center}
\end{center}

\bigskip

The next question which comes to mind is the following:

\begin{ques}
Does the Diophantine equation
$$
x(x+1)y(y+1)=z^3
$$
have infinitely many solutions in positive integers?
\end{ques}

In the range $x\leq y\leq 10^5$ we found only three solutions:
$$
(x,y,z)=(11,242,198),\; (32, 242, 396),\;(539, 3024, 13860).
$$

We were trying to prove that relatively simpler Diophantine equation
\begin{equation}\label{last}
x(x+1)y(y+1)z(z+1)=t^3
\end{equation}
has infinitely many solutions in positive integers satisfying the condition $x+1<y<z-1$, but we failed.  We find that in the range $x+1<y<z-1<5\cdot 10^3$ our equation has 88 solutions.
The list of solutions can be downloaded from \url{http://math.unideb.hu/media/tengely-szabolcs/xyzt3.pdf}.
This strongly suggests that the following is true:
\begin{conj}
The equation (\ref{last}) has infinitely many solutions in positive integers satisfying the condition $x+1<y<z-1$.
\end{conj}

\begin{acknowledgement}
We express our gratitude to the anonymous referee for a careful reading of the manuscript and valuable suggestions made.
\end{acknowledgement}

\end{document}